\documentclass[oneside,11pt]{amsart}
\usepackage{geometry}
\usepackage{parskip}
\usepackage{hyperref}
\usepackage{amssymb}
\usepackage{amsmath}
\usepackage{amsthm}
\usepackage{mathtools}
\usepackage{tikz}
\usepackage{tikz-cd}
\usepackage{amsfonts}
\usepackage{amscd}
\usetikzlibrary{matrix,arrows,shapes.geometric}

\allowdisplaybreaks

\newtheorem*{thma}{Theorem A}
\newtheorem*{thmb}{Theorem B}
\newtheorem*{thmc}{Theorem C}
\newtheorem*{thma'}{Theorem A*}
\newtheorem*{thmc'}{Theorem C*}
\newtheorem{thm}{Theorem}[section]
\newtheorem{defn}[thm]{Definition}
\newtheorem{prop}[thm]{Proposition}
\newtheorem{remk}[thm]{Remark}
\newtheorem{conj}[thm]{Conjecture}
\newtheorem{lem}[thm]{Lemma}
\newtheorem{cor}[thm]{Corollary}

\newcommand{\Z}{\mathbb Z}

\newcommand{\R}{\mathbb R}
\newcommand{\Q}{\mathbb Q}

\title[Generators of Bieberbach groups]{Generators of Bieberbach groups with 2-generated holonomy group}
\author{Ho Yiu CHUNG}


\date{\today}

\begin{document}
\begin{abstract} An $n$-dimensional Bieberbach group is the fundamental group of a closed flat $n$-dimensional manifold. K. Dekimpe and P. Penninckx conjectured that an $n$-dimensional Bieberbach group can be generated by $n$ elements. 
In this paper, we show that the conjecture is true if the holonomy group is 2-generated (e.g. dihedral group, quaternion group or simple group) or the order of holonomy group is not divisible by 2 or 3. 
In order to prove this, we show that an $n$-dimensional Bieberbach group with cyclic holonomy group of order larger than two can be generated by $(n-1)$ elements.
\end{abstract}

\maketitle

\vspace{-3mm}

\section{Introduction}
We first introduce the geometric definition of a crystallographic groups. A group $\Gamma$ is said to be an {\it $n$-dimensional crystallographic group} if it is a discrete subgroup of $\mathbb{R}^n\rtimes O(n)$, which is the group of isomotries of $\mathbb{R}^n$ and it acts cocompactly on $\mathbb{R}^n$. 
By The First Bieberbach Theorem, \cite[Theorem 2.1]{cry}, $\Gamma\cap(\R^n\times I)$ is isomorphic to $\mathbb{Z}^n$ and $\Gamma/\Gamma\cap(\mathbb{R}^n\times I)$ is a finite group called the {\it holonomy groups of $\Gamma$}.
We say $\Gamma$ is an {\it $n$-dimensional Bieberbach group} if it is an $n$-dimensional torsion-free crystallographic group.
In this paper, we focus on the below conjecture.
\begin{conj}\cite[Dekimpe-Penninckx]{nm gen bieb}
	{\normalfont Let $\Gamma$ be an $n$-dimensional Bieberbach group. Then the minimum number of generators of $\Gamma$ is less than or equal to $n$.} 
\end{conj}
The conjecture was solved for some special cases. 
For example, the conjecture is true if the holonomy group is an odd prime $p$-group (see \cite{nan}), or the holonomy group is an elementary abelian $p$-group (see \cite{nm gen bieb}). 
On the other hand, by a computer program namely CARAT, it has been checked that the conjecture is true if the Bieberbach group has dimension less than 7 (see \cite{carat}). 

There is a connection between the number of generators of Bieberbach group and the number of generators of a finite group that can act freely on an $n$-torus (see \cite{vieira}). Let $G$ be a finite group. If $G$ acts freely on an $n$-torus $T^n$, the quotient space $T^n/G$ is a manifold $M$ and we get a short exact sequence as below,
\[
\begin{tikzpicture}[node distance=2cm, auto]
\node (GA) {$\pi_1(M)$};
\node (G) [right of=GA] {$G$};
\node (I) [right of=G] {$1$};
\node (Z) [left of=GA] {$\pi_1(T^n)$};
\node (O) [left of=Z] {$0$};
\draw[->] (GA) to node {} (G);
\draw[->] (G) to node {}(I);
\draw[->] (O) to node {}(Z);
\draw[->] (Z) to node {}(GA);	 
\end{tikzpicture}
\]
where $\pi_1(M)$ is an $n$-dimensional Bieberbach group. Hence if $\pi_1(M)$ can be generated by $n$ elements, then the minimal number of generators of $G$ should not be larger than $n$. For instance, we know that $(\Z/2\Z)^{n+1}$ cannot act freely on $T^n$ for $n\geq1$. (see \cite{nm gen bieb},\cite{vieira}). 

Let $G$ be a group and $M$ be a $\mathbb{Z}G$-module. Throughout this paper, we denote $d(G)$ to be the minimal number of generators of $G$ and denote $rk_G(M)$ to be the minimal number of generators of $M$ as a $\mathbb{Z}G$-module.
Our paper is divided into several sections. In Section 2, we give some basic definitions and some related properties of crystallographic groups. 
In Section 3, we discuss the number of generators of $\Z C_m$-module, where $C_m$ is a cyclic group of order $m$. In Section 4, we present our three main theorems. The below three theorems are our main results.

\begin{thma}{\normalfont Let $\Gamma$ be an $n$-dimensional crystallographic group with holonomy group isomorphic to $C_m=\langle g|g^m=1\rangle$ where $m\geq 3$.

($i$) If $m$ is divisible by prime larger than 3, then $d(\Gamma)\leq n-2$.

($ii$) If $m$ is not divisible by prime larger than 3 and $\Gamma$ is torsion-free, then $d(\Gamma)\leq n-1$.}
\end{thma}

 The idea of the proof of Theorem A($i$) is to consider $\Gamma\cap (\mathbb{R}^n\times I)$ as a $\mathbb{Z}C_p$-module where $p$ is prime larger than 3. We use the module structure to reduce the number of generators. For Theorem A($ii$), we construct a surjective homomorphism from $\Gamma$ to $\Z$. Then by studying how $\Z$ acts on the kernel of the homomorphism, we can eliminate some redundant generators. 
 
 By Theorem A, we get two corollaries. One shows that a general $n$-dimensional Bieberbach group can be generated by $2n$ elements. The other corollary shows an $n$-dimensional Bieberbach group with a simple group as holonomy group can be generated by $n-1$ elements.

\begin{thmb}{\normalfont Let $\Gamma$ be an $n$-dimensional crystallographic group with holonomy group isomorphic to a finite group $G$, where the order of $G$ is not divisible by 2 or 3. Then $d(\Gamma)\leq n$.}
\end{thmb}

The idea of the proof of Theorem B is to apply results from \cite{gen of fin gp} to get a relation between the number of generators of the finite group $G$ and its Sylow $p$-subgroups.
Then we apply results from \cite{nan} to prove Theorem B. 

\begin{thmc}{\normalfont Let $\Gamma$ be an $n$-dimensional Bieberbach group with $2$-generated holonomy group. Then $d(\Gamma)\leq n$}.
\end{thmc}

The idea of the proof of Theorem C is to consider a Bieberbach subgroup with cyclic holonomy group. 
Then we apply Theorem A to get the desired bound for generators of the Bieberbach group $\Gamma$.

\subsection{Acknowledgment}

I would like to thank my supervisor Dr. Nansen Petrosyan for his help and guidance.

\section{Background}
	In this section, we recall some properties of crystallographic group from \cite{charlap} and \cite{cry}. Let \(\Gamma\) be an \(n\)-dimensional crystallographic group. By The First Bieberbach Theorem, \cite[Theorem 2.1]{cry}, $\Gamma\cap(\mathbb{R}^n\times I)$ is isomorphic to $\mathbb{Z}^n$ and it is the maximal abelian subgroup with finite index, where $I$ is the identity element in the orthogonal group. Therefore $\Gamma$ can be expressed as the short exact sequence
	\begin{equation}
	\begin{tikzpicture}[node distance=1.5cm, auto]
	\node (GA) {$\Gamma$};
	\node (G) [right of=GA] {$G$};
	\node (I) [right of=G] {$1$};
	\node (Z) [left of=GA] {$\mathbb{Z}^n$};
	\node (O) [left of=Z] {$0$};
	\draw[->] (GA) to node {$p$} (G);
	\draw[->] (G) to node {}(I);
	\draw[->] (O) to node {}(Z);
	\draw[->] (Z) to node {$\iota$}(GA);	 
	\end{tikzpicture}
	\end{equation}
	where $G$ is a finite group, $\iota:\mathbb{Z}^n\hookrightarrow\Gamma$ is an inclusion map which maps $e_i$ to $(e_i,I)$ where $e_1,...,e_n$ are the standard basis of $\mathbb{Z}^n$ and \(p:\Gamma\rightarrow G\) is a projection map which maps $(a,A)$ to $A$. Given such a short exact sequence, it will induce a representation $\rho:G\rightarrow GL_n(\mathbb{Z})$ given by $\rho(g)x=\bar{g}\iota(x)\bar{g}^{-1}$, where $x\in\mathbb{Z}^n$ and $\bar{g}$ is chosen arbitrarily such that $p(\bar{g})=g$. In this case, we call the group $G$ to be the {\it holonomy group} and the representation $\rho$ to be the {\it holonomy representation} of $\Gamma$. It is well known that $\rho$ is a faithful representation (see \cite[Chapter 2]{cry}). 
	
	Now we are going to introduce the algebraic definition for crystallographic groups, which is equivalent to the geometric definition of crystallographic groups (see \cite[Theorem 2.2]{cry}). We say $\Gamma$ is an {\it $n$-dimensional crystallographic group} if it can be expressed as the below short exact sequence
	\begin{equation}
	\begin{tikzpicture}[node distance=1.5cm, auto]
	\node (GA) {$\Gamma$};
	\node (G) [right of=GA] {$G$};
	\node (I) [right of=G] {$1$};
	\node (Z) [left of=GA] {$\mathbb{Z}^n$};
	\node (O) [left of=Z] {$0$};
	\draw[->] (GA) to node {} (G);
	\draw[->] (G) to node {}(I);
	\draw[->] (O) to node {}(Z);
	\draw[->] (Z) to node {}(GA);	 
	\end{tikzpicture}
	\end{equation}
	where $G$ is finite group and the induce representation $\rho:G\rightarrow GL_n(\mathbb{Z})$ is a faithful representation. 
	Given an $n$-dimensional crystallographic group $\Gamma$, with holonomy group $G$. Every element $\gamma\in\Gamma$ can be expressed as $(a,g)$ where $a\in\mathbb{R}^n$ and $g\in G$. 
	The operation in $\Gamma$ is given by $(a_1,g_1)(a_2,g_2)=\left(a_1+\rho(g_1)(a_2),g_1g_2\right)$, where $(a,g),(b,h)\in\Gamma$ and $\rho$ is the holonomy representation.
	Notice that $\Gamma$ will induce the holonomy representation $\rho:G\rightarrow GL_n(\mathbb{Z})$. 
	Therefore we can consider $\Gamma\cap(\mathbb{R}^n\times I)\cong\mathbb{Z}^n$ as a $\mathbb{Z}G$-module. 
	We denote $rk_G(\mathbb{Z}^n)$ to be the minimal number of generators of $\mathbb{Z}^n$ as a $\mathbb{Z}G$-module. 
	In particular, if $G$ is a cyclic group with generator $g$ and let $\rho:G\rightarrow GL_n(\mathbb{Z})$ where $g\mapsto M\in GL_n(\mathbb{Z})$ be its matrix holonomy representation. 
	For convenience, we denote element $(a,g)\in\Gamma$ to be $(a,M)$ and denote the $\Z G$-module $\mathbb{Z}^n$ to be $\mathbb{Z}^n_M$ to specify that the $G$-action is given by the matrix $M$.
	We will denote $I_n$ to be the identity matrix of dimension $n$ and $C_m$ to be a cyclic group of order $m$.
	\begin{remk}{\label{rmk gen}}
		{\normalfont Let $\Gamma$ be an $n$-dimensional crystallographic group with holonomy group $G$, where $G$ is generated by $m$ elements namely $a_1,...,a_m$. Then by sequence (1), we have the following two observations,
		
		($i$) $d(\Gamma)\leq rk_G(\mathbb{Z}^n)+d(G)$.
		
		($ii$) $\{\iota(e_1),...,\iota(e_n),\alpha_1,...,\alpha_m\}$ can be a generating set of $\Gamma$ where $e_1,...,e_n$ are the standard basis of $\mathbb{Z}^n$ and $\alpha_i$ is chosen arbitrarily such that $p(\alpha_i)=a_i$ for all $i=1,...,m$.}
	\end{remk}
	
	\begin{defn}
		{\normalfont Let $G$ be a group, $M$ be a $\mathbb{Z}G$-module and $\rho:G\rightarrow GL_m(\mathbb{Z})$ be the representation correspond to the $\mathbb{Z}G$-module $M$.
		
		($i$) $N$ is a {\it submodule} of $M$ if $N$ is a subgroup of $M$ which is closed under the action of ring elements.
		
		($ii$) $M$ is decomposable if $M$ is the direct sum of submodules. $M$ is indecomposable if $M$ is not decomposable.
		
		($iii$) $M$ is {\it $\mathbb{Z}$-reducible} if there exists a matrix $N\in GL_m(\Z)$ such that $N\rho(g) N^{-1}=\begin{pmatrix}
		* & * \\
		0 & *
		\end{pmatrix}$ for $g\in G$.
		$M$ is {\it $\mathbb{Z}$-irreducible} if $M$ is not $\mathbb{Z}$-reducible.}
	\end{defn}
	Now, we are going to give a short introduction to the properties of holonomy representation. Let $M_1,...,M_k$ be square matrices with entries in $\mathbb{Z}$, we denote $tri(M_1,...,M_k)$ to be matrix of form as below,
	\[tri(M_1,...,M_k):=\begin{pmatrix}
	M_1 & & & *\\
	& M_2 & & \\
	& & \ddots & \\
	\mbox{\Huge 0} & & & M_k \\
	\end{pmatrix}\]
	Let $\rho:C_m\rightarrow GL_n(\mathbb{Z})$ be a faithful representation. 
	Since $\rho$ is defined up to isomorphism, we are able to conjugate it by a suitable invertible matrix and assume  $\rho(g)=tri(A_1,...,A_t)$ for some $t\in\mathbb{N}$ and $A_1,...,A_t$ are square matrices such that $\Z^{dim(A_1)}_{A_1}$,..., $\Z^{dim(A_t)}_{A_t}$ are $\Z$-irreducible modules and $\sum_{j=1}^{t}dim(A_j)=n$.
	\begin{remk}{\label{rmk ord}}
		{\normalfont Let $M=tri(A_1,...,A_t)$ where $A_1,...,A_t$ are square matrices. Denote the order of $A_i$ to be $a_i$ for $i=1,...,t$ and $m$ to be the order of $M$. Then the least common multiple of $a_1,...,a_t$ equals to $m$. In particular, $m$ is divisible by $a_i$ for $i=1,...,t$.}
	\end{remk}

	\section{Generators of $\Z C_m$-module}
	
	Let $\Gamma$ be an $n$-dimensional crystallographic group with holonomy group isomorphic to $C_m$. We can consider $\Gamma\cap(\mathbb{R}^n\times I)\cong\mathbb{Z}^n$ as a $\mathbb{Z}C_m$-module. Since we can restrict the $C_m$-action to be a $C_k$-action as long as $m$ is divisible by $k$, we can also view $\Z^n$ as a $\Z C_k$-module. It is clear that $rk_{C_m}(\Z^n)\leq rk_{C_k}(\Z^n)$. The below lemma and proposition are on the number of generators of $\Z C_m$-module.
	
	\begin{lem}{\label{rank c_p}}
		{\normalfont Let $\rho:C_p\rightarrow GL_n(\Z)$ be a faithful representation and $\Z^n$ be the correspondence $\Z C_p$-module, where $p$ is prime. Then $rk_{C_p}(\Z^n)\leq n-p+a$, where $a=2$ if $p\leq 19$, otherwise $a=3$.}
	\end{lem}
	\begin{proof}
		Let $g$ be the generator of $C_p$. Assume $\rho(g)=tri(A_1,...,A_k)$ where $\Z^{dim(A_1)}_{A_1}$,$\cdots$, $\Z^{dim(A_k)}_{A_k}$ are $\Z$-irreducible $\Z C_p$-modules. By Remark \ref{rmk ord}, there exists $i\in\{1,...,k\}$ such that $A_i$ has order $p$. By \cite[section 74]{repn} $A_i$ has dimension $p-1$ and the module $\Z^{dim(A_i)}_{A_i}$ is isomorphic to an ideal in $\Z[\zeta]$ where $\zeta$ is a primitive $p$-root of unity. If $p\leq 19$, by \cite[Section 29.1.3]{alg nm}, the class number of $\Z[\zeta]$ is 1. Therefore the module $\Z^{dim(A_i)}_{A_i}$ is a principle ideal and it is isomorphic to $\Z[\zeta]$. Hence $rk_{C_p}(\Z^{dim(A_i)}_{A_i})=1$. Now assume $p>19$. Since $\Z[\zeta]$ is a Dedekind domain. By \cite[Section 7.1-2]{alg nm}, every ideal in a Dedekind domain can be generated by two elements. Hence $rk_{C_p}(\Z^{dim(A_i)}_{A_i})\leq 2$. Therefore we have
		\[rk_{C_p}(\Z^n)\leq n-dim(A_i)+rk_{C_p}(\Z^{dim(A_i)}_{A_i})=n-p+1+rk_{C_p}(\Z^{dim(A_i)}_{A_i})\leq n-p+a\]
		where $a=2$ if $p\leq 19$, otherwise $a=3$.
	\end{proof}
	
	\begin{prop}{\label{rank c_m}}
		{\normalfont Let $\rho:C_m\rightarrow GL_n(\Z)$ be a faithful representation and $\Z^n$ be the correspondence $\Z C_m$-module of $\rho$, where $m\geq 3$.
			
		($i$) If $m$ is divisible by prime larger than 3, then $rk_{C_m}(\Z^n)\leq n-3$.
		
		($ii$) If $m$ is not divisible by prime larger than 3, then $rk_{C_m}(\Z^n)\leq n-1$.}
	\end{prop}
	\begin{proof}
		Let $m=p_1^{s_1}\cdots p_t^{s_t}$ be the prime decomposition of $m$ and assume $p_1<\cdots<p_t$. Let $g$ be the generator of $C_m$.
		
		($i$): Consider $H=\langle g^{m/p_t}\rangle\cong C_{p_t}$, a subgroup of $C_m$. We can view $\mathbb{Z}^n$ as a $\mathbb{Z}C_{p_t}$-module where the $C_{p_t}$-action is given by $\rho|_H$. Since $\rho|_H$ is a faithful representation, by Lemma \ref{rank c_p}, we have $rk_{C_{p_t}}(\Z^n)\leq n-p_t+a$. Since $m$ is divisible by prime larger than 3, we have $rk_{C_m}(\mathbb{Z}^n)\leq n-3$. 
		
		($ii$): We observe that $m$ is either divisible by 3 or 4. If $m$ is divisible by 3, we consider $\Z^n$ as $\Z C_3$-module. By Lemma \ref{rank c_p}, we have $rk_{C_3}(\Z^n)\leq n-1$. Hence $rk_{C_m}(\Z^n)\leq n-1$. Now we assume $m$ is divisible by 4. Consider $H'=\langle g^{m/4}\rangle\cong C_4$, a subgroup of $C_m$. We can view $\Z^n$ as a $\Z C_4$-module by restricting the $C_m$-action to a $C_4$-action, where the $C_4$-action is given by $\rho|_{H'}$. We assume $\rho|_{H'}(g^{m/4})=tri(M_1,...,M_k)$ where $\Z^{dim(M_1)}_{M_1}$,..., $\Z^{dim(M_k)}_{M_k}$ are $\Z$-irreducible $\Z C_4$-modules. By Remark \ref{rmk ord}, there exists $i\in\{1,...,k\}$ such that $M_i$ is a matrix of order 4. Let $\phi:C_4\rightarrow GL_n(\Z)$ be the corresponding representation of $\Z^{dim(M_i)}_{M_i}$. By \cite[Section 5]{nan c4}, there is only one faithful integral $\Z$-irreducible $C_4$-representation up to equivalence. Hence we assume $M_i$ is equivalent to $\begin{pmatrix}
		0 & 1\\-1 & 0
		\end{pmatrix}$. Let $y_1=(1,0)\in\Z^2$ and $y_2=(0,1)\in\Z^2$ be the standard basis of $\mathbb{Z}^2_{M_i}$. We have $\phi(g^{m/4})y_2=y_1$. Hence $\mathbb{Z}^2_{M_i}$ can be generated by $y_2$ as a $\mathbb{Z}C_4$-module. Since $rk_{C_4}(\Z^{m_i}_{M_i})\leq m_i-1$, we have $rk_{C_m}(\Z^n)\leq\sum_{z=1}^{k}rk_{C_4}(\Z^{m_z}_{M_z})\leq n-1$.
	\end{proof}
	
	The following Lemma is on the dimension of the fix point set of a cyclic holonomy representation.
	
	\begin{lem}{\label{betti}}
		{\normalfont Let $\Gamma$ be an $n$-dimensional Bieberbach group with holonomy group isomorphic to $C_m=\langle x|x^m=1\rangle$. Then $(\mathbb{Z}^n)^{C_m}\neq 0$.}
	\end{lem}
	\begin{proof}
		Let $\alpha\in H^2(C_m,\mathbb{Z}^n)$ be the second cohomology class that corresponds to the extension (1) of $\Gamma$. Assume by contradiction that $(\mathbb{Z}^n)^{C_m}=0$. By \cite[Page 58]{Brown}, we have
		\[H^2(C_m,\mathbb{Z}^n)=(\mathbb{Z}^n)^{C_m}/\{(1+x+x^2+...+x^{m-1})z|z\in\mathbb{Z}^n\}\]
		Given $(\mathbb{Z}^n)^{C_m}=0$, hence we have $H^2(C_m,\mathbb{Z}^n)=0$, which forces $\alpha=0$. By \cite[Theorem 3.1]{cry}, since $\alpha=0$, that means the extension (1) of $\Gamma$ splits and therefore $\Gamma$ has torsion. Hence $\Gamma$ is not a Bieberbach group, which is a contradiction.
	\end{proof}

	\section{Main Result}
	
	\begin{thma}
		{\normalfont Let $\Gamma$ be an $n$-dimensional crystallographic group with holonomy group isomorphic to $C_m=\langle g|g^m=1\rangle$ where $m\geq 3$.
			
			($i$) If $m$ is divisible by prime larger than 3, then $d(\Gamma)\leq n-2$.
			
			($ii$) If $m$ is not divisible by prime larger than 3 and $\Gamma$ is torsion-free, then $d(\Gamma)\leq n-1$.}
	\end{thma}
	\begin{proof}
		($i$): By Remark \ref{rmk gen}, we have $d(\Gamma)\leq rk_{C_m}(\Z^n)+1$. Since $m$ is divisible by prime larger than 3, by Proposition \ref{rank c_m}, we have $rk_{C_m}(\Z^n)\leq n-3$. Therefore we have $d(\Gamma)\leq n-2$.
		
		($ii$): By Remark \ref{rmk gen}, let $\Gamma=\langle \iota(e_1),...,\iota(e_n),\alpha\rangle$, where $e_1,...,e_n$ are the standard basis of $\Z^n$ and $p(\alpha)=g$. By \cite[Proposition 1.4]{LA} and \cite[Lemma 5.2]{cry}, we have $b_1(\Gamma)=rk((\Z^n)^{C_m})$. By Lemma \ref{betti}, let $k=b_1(\Gamma)>0$. Without loss of generality, every element of $\Gamma$ can be expressed as $\left(a,tri(M,I_k)\right)$ where $a\in\R^n$ and $M$ is a square matrix of dimension $n-k$. In particular, let $\alpha=\left(x,tri(A,I_k)\right)$ where $x=(x_1,...,x_n)\in\R^n$ and $A$ is a square matrix of dimension $n-k$ which do not fix any non-trivial elements. In other words, $Au=u$ if and only if $u=0$ for $u\in\R^{n-k}$. First we assume $x_{n-k+1}=\cdots=x_n=0$. Let $v:=(x_1,...,x_{n-k})\in\R^{n-k}$. By simple calculations, we get $\alpha^m=\left((\sum_{s=0}^{m-1}A^sv,0,...,0),I_n\right)$. Since $A(\sum_{s=0}^{m-1}A^sv)=\sum_{s=0}^{m-1}A^sv$, we have $\sum_{s=0}^{m-1}A^sv=0$. There is a contradiction because $\alpha^m=(0,I_n)$. Therefore there exists $i\in\{n-k+1,...,n\}$ such that $x_i=\frac{q}{z}\neq 0\in\Q$. Define $f:\Gamma\rightarrow\Z$ where it maps $\left((y_1,...,y_n),tri(M,I_k)\right)\in\Gamma$ to $zy_i\in\Z$. Hence we have $f(\alpha)=q$, $f(\iota(e_i))=z$ and $f(\iota(e_j))=0$ for all $j\neq i$. We claim that $f$ is a surjective homomorphism. Let $\gamma_1=((m_1,...,m_n),tri(M_1,I_k))\in\Gamma$ and $\gamma_2=((\bar{m_1},...,\bar{m_n}),tri(M_2,I_k))\in\Gamma$. By simple calculation, we get $\gamma_1\gamma_2=((*,...,*,m_{n-k+1}+\bar{m}_{n-k+1},...,m_n+\bar{m}_{n}),tri(M_1M_2,I_k))$. Hence we have $f(\gamma_1)+f(\gamma_2)=f(\gamma_1\gamma_2)$. Therefore $f$ is a homomorphism. Notice that $q$ and $z$ are coprime, there exists integers $\sigma$ and $\tau$ such that $\sigma q+\tau z=1$. Hence we have $f(\alpha^\sigma\iota(e_i)^\tau)=1$. Therefore $f$ is surjective. Observe that $ker(f)=\langle\iota(e_1),...,\hat{\iota(e_i)},...,\iota(e_n)\rangle\cong\Z^{n-1}$. We have the below short exact sequence
		\begin{equation}
		\begin{tikzpicture}[node distance=2.5cm, auto]
		\node (GA) {$\Gamma$};
		\node (G) [right of=GA] {$\Z$};
		\node (I) [right of=G] {$0$};
		\node (Z) [left of=GA] {$ker(f)\cong\mathbb{Z}^{n-1}$};
		\node (O) [left of=Z] {$0$};
		\draw[->] (GA) to node {$f$} (G);
		\draw[->] (G) to node {}(I);
		\draw[->] (O) to node {}(Z);
		\draw[->] (Z) to node {}(GA);	 
		\end{tikzpicture}
		\end{equation}
		By \cite[Chapter IV, Section 1]{Brown}, such short exact sequence will induce a representation $\rho:\Z\rightarrow GL_{n-1}(\Z)$ given by $\rho(x)e_j=\bar{x}\iota(e_j)\bar{x}^{-1}$ where $f(\bar{x})=x$ for all $j\neq i$. Consider the restriction $\bar{\rho}:=\rho|_{q\Z}:q\Z\rightarrow GL_{n-1}(\Z)$. We claim that $ker(\bar{\rho})=mq\Z$. Let $qx\in ker(\bar{\rho})$ for any $x\in\Z$. We have $e_j=\bar{\rho}(qx)e_j=\alpha^x\iota(e_j)\alpha^{-x}=p(\alpha^x)e_j$ for all $j\neq i$. Hence $p(\alpha^x)$ needs to be an identity matrix. Therefore $x$ is multiple of $m$ or $x=0$. Hence $ker(\bar{\rho})\subseteq mq\Z$. Since $p(\alpha^m)$ is an identity matrix, $\rho(mqx)(e_j)=\alpha^{mx}\iota(e_j)\alpha^{-mx}=p(\alpha^{mx})e_j=e_j$ for all $j\neq i$ and $x\in\Z$. Hence $mq\Z\subseteq ker(\bar{\rho})$. Therefore we have $ker(\bar{\rho})=mq\Z$. Now we can obtain a faithful representation $\phi:a\Z/ma\Z\rightarrow GL_{n-1}(\Z)$ given by $\phi(\bar{x})=\bar{\rho}(x)$ where $x$ is the representative of $\bar{x}\in a\Z/ma\Z$. Hence we can view $\Z^{n-1}$ as a $\Z C_m$-module with faithful $C_m$-representation. By Proposition \ref{rank c_m}, $\Z^{n-1}$ can be generated by $n-2$ elements. By (3), we have $d(\Gamma)\leq rk_{C_m}(\Z^{n-1})+1\leq n-1$.   
	\end{proof}
	The corollary below gives the general bound on the number of generators of general Bieberbach groups.
	\begin{cor}{\label{bdd}}
		{\normalfont Let $\Gamma$ be an $n$-dimensional Bieberbach group with holonomy group $G$. Then $d(\Gamma)\leq 2n$.}
	\end{cor}
	\begin{proof}
		Let $|G|=p_1^{s_1}\cdots p_k^{s_k}$ be the prime decomposition of order of $G$. By \cite[Theorem A]{gen of fin gp}, we have
		\[d(G)\leq \max_{1\leq i\leq k}d(P_i)+1\]
		where $P_i$ is the Sylow $p_i$-subgroup of $G$ for $i=1,...,k$. We fix $j\in\{1,...,k\}$ such that $d(P_j)=\max_{1\leq i\leq k}d(P_i)$. We first assume $p_j\geq 3$. We can consider $\Gamma\cap(\R^n\times I)\cong\Z^n$ as a $\Z P_j$-module. By \cite[Theorem A]{nan}, we have $d(P_j)+rk_{P_j}(\Z^n)\leq n$. Hence we have $d(\Gamma)\leq d(G)+1+rk_{P_j}(\Z^n)\leq n+1$. Now we assume $p_j=2$. If $G$ is a $2$-group, then by \cite[Theorem A]{nan}, we have $d(\Gamma)\leq 2n$. If $G$ is not a 2-group, then there exists $g\in G$ such that $g$ has order $p\geq 3$. Hence we can consider $\Z^n$ as a $\Z C_p$-module. By Lemma \ref{rank c_p}, we have $rk_{C_p}(\Z^n)\leq n-1$. By \cite[Proposition 2.2]{nan}, we have $d(P_j)\leq n$. Hence we have $d(\Gamma)\leq 2n$.
	\end{proof}

	\begin{cor}{\label{simple}}
		{\normalfont Let $\Gamma$ be an $n$-dimensional Bieberbach group with holonomy group $G$, where $G$ is a simple group but not $C_2$. Then $d(\Gamma)\leq n-1$.}
	\end{cor}
	\begin{proof}
		By Remark \ref{rmk gen}, we have $d(\Gamma)\leq d(G)+rk_G(\Z^n)$. If $G$ is a cyclic group of odd prime order, then by Theorem A, we have $d(\Gamma)\leq n-1$. If $G$ is not cyclic, by Burnside's Theorem, \cite[Page 886]{dummit}, there exists a prime $p\geq 5$ such that the order of $G$ is divisible by $p$. So we can view $\Z^n$ as a $\Z C_p$-module. By Lemma \ref{rank c_p}, we have $rk_{C_p}(\Z^n)\leq n-p+a\leq n-3$, where $a=2$ if $p\leq 19$, otherwise $a=3$. By \cite[Theorem B]{simple}, we have $d(G)\leq 2$. Hence we have $d(\Gamma)\leq d(G)+rk_G(\Z^n)\leq 2+rk_{C_p}(\Z^n)\leq n-1$.
	\end{proof}
	The rest of the paper will present the proof of Theorem B and Theorem C.
	\begin{thmb}{\normalfont Let $\Gamma$ be an $n$-dimensional crystallographic group with holonomy group isomorphic to a finite group $G$, where the order of $G$ is not divisible by 2 or 3. Then $d(\Gamma)\leq n$.}
	\end{thmb}
	\begin{proof}
		Let $|G|=p_1^{s_1}\cdots p_k^{s_k}$ be the prime decomposition of the order of $G$. First, we want to calculate the number of generators of the holonomy group $G$. By \cite[Theorem A]{gen of fin gp}, we have
		\[d(G)\leq \max_{1\leq i\leq k}d(P_i)+1\]
		where $P_i$ is the Sylow $p_i$-subgroup of $G$ for $i=1,...,k$. We fix $j\in\{1,...,k\}$ such that $d(P_j)=\max_{1\leq i\leq k}d(P_i)$. Let $\rho:G\rightarrow GL_n(\mathbb{Z})$ be the holonomy representation for $\Gamma$. By definition, $\rho$ is a faithful representation. Therefore $P_i$ acts faithfully on $\mathbb{Z}^n$. By \cite[Proposition 2.2]{nan}, we have 
		\[d(G)\leq \frac{n-rk\left((\mathbb{Z}^n)^{P_j}\right)}{p_j-1}+1\]
		Now, we consider the lattice part. We can view $\Gamma\cap(\mathbb{R}^n\times I)\cong\mathbb{Z}^n$ as a $\mathbb{Z}P_j$-module. By \cite[Proposition 2.5]{nan}, we have 
		\[rk_{P_j}(\mathbb{Z}^n)\leq \frac{(a-1)\left(n-rk(\mathbb{Z}^n)^{P_j}\right)}{p_j-1}+rk(\mathbb{Z}^n)^{P_j}\]
		where $a=2$ if $p_j\geq 19$, otherwise $a=3$. Therefore we have
		\begin{align*}
		d(\Gamma)\leq d(G)+rk_{P_j}(\mathbb{Z}^n)&\leq \frac{n-rk\left((\mathbb{Z}^n)^{P_j}\right)}{p_j-1}+1+\frac{(a-1)\left(n-rk(\mathbb{Z}^n)^{P_j}\right)}{p_j-1}+rk(\mathbb{Z}^n)^{P_j}\\
		&=\frac{a\left(n-rk(\mathbb{Z}^n)^{P_j}\right)}{p_j-1}+rk(\mathbb{Z}^n)^{P_j}+1
		\end{align*}
		We need to show
		\[
		\frac{a\left(n-rk(\mathbb{Z}^n)^{P_j}\right)}{p_j-1}+rk(\mathbb{Z}^n)^{P_j}+1\leq n
		\]
		We have
		\begin{align*}
		&\frac{a\left(n-rk(\mathbb{Z}^n)^{P_j}\right)}{p_j-1}+rk(\mathbb{Z}^n)^{P_j}+1\leq n\\
		\iff & an-a\cdot rk(\mathbb{Z}^n)^{P_j}+(p_j-1)rk(\mathbb{Z}^n)^{P_j}+p_j-1\leq n(p_j-1)\\
		\iff & (p_j-1-a)rk(\mathbb{Z}^n)^{P_j}\leq (p_j-1-a)n-(p_j-1)\\
		\iff & rk(\mathbb{Z}^n)^{P_j}\leq n-\frac{p_j-1}{p_j-1-a}=n-1-\frac{a}{p_j-1-a}
		\end{align*}
		If $5\leq p_j\leq 19$, we have $\frac{a}{p_j-1-a}=\frac{2}{p_j-3}\leq 1$. If $p_j>19$, we have $\frac{a}{p_j-1-a}=\frac{3}{p_j-4}<1$. Therefore we can conclude that if $rk(\mathbb{Z}^n)^{P_j}\leq n-2$, then $d(\Gamma)\leq n$. By Cauchy's Theorem  \cite[Page 93]{dummit}, $P_j$ has an element $x\in P_j$ with order $p_j$. Let $C_{p_j}$ be a cyclic subgroup of $P_j$ generated by $x$. Consider $(\mathbb{Z}^n)^{C_{P_j}}$, where $C_{P_j}$ acts faithfully on $\mathbb{Z}^n$ via $\rho|_{C_{P_j}}:C_{P_j}\rightarrow GL_n(\mathbb{Z})$. By \cite[Section 73]{repn}, the degree of a faithful indecomposable $C_{p_j}$-representation is either $p_j-1$ or $p_j$. If the degree is $p_j-1$, then it has trivial fix point set. If the degree is $p_j$, then the fix point set is 1-dimensional. Observe that $rk(\mathbb{Z}^n)^{C_{p_j}}$ has maximum value when $\rho|_{C_{P_j}}$ is a direct sum of one faithful indecomposable sub-representation and all others are trivial sub-representations. Therefore $rk(\mathbb{Z}^n)^{C_{p_j}}\leq n-p_j+1\leq n-4$. Hence we have $rk(\mathbb{Z}^n)^{P_j}\leq n-4$. Therefore we can conclude $d(\Gamma)\leq n$.
	\end{proof}

	\begin{thmc}{\normalfont Let $\Gamma$ be an $n$-dimensional Bieberbach group with $2$-generated holonomy group. Then $d(\Gamma)\leq n$.}
	\end{thmc}
	\begin{proof}
		Let $G$ be the holonomy group of $\Gamma$. Let $x$ and $y$ be the generators of $G$. They have order $a$ and $b$ respectively. If either $a=1$ or $b=1$, then $G$ is a cyclic group. By \cite[Theorem 5.7]{nm gen bieb} and Theorem A, $d(\Gamma)\leq n$. Next, consider cases where $a\geq 3$ or $b\geq 3$. It is sufficient to consider only the case where $a\geq 3$. By Remark \ref{rmk gen}, let $\Gamma=\langle \iota(e_1),...,\iota(e_n),\alpha,\beta\rangle$, where $e_1,...,e_n$ are the standard basis for $\mathbb{Z}^n$, $p(\alpha)=x$ and $p(\beta)=y$. Define $\Gamma'=\langle \iota(e_1),...,\iota(e_n),\alpha\rangle$. Notice that $\Gamma'$ is an $n$-dimensional Bieberbach subgroup of $\Gamma$ with holonomy group $C_a$. Since $a\geq 3$, by Theorem A, $d(\Gamma')\leq n-1$. 
		Hence we have $d(\Gamma)\leq n$. Finally, we assume $a=b=2$. Consider element $xy\in G$. Since $G$ is finite, $xy$ has finite order. If $xy$ is of order 1 (i.e. $xy=1$), then $x=y$. So $G\cong C_2$. By \cite[Theorem 5.7]{nm gen bieb}, $d(\Gamma)\leq n$. If $xy$ is of order 2 (i.e. $xyxy=1$), then $xy=yx$. Hence $G\cong C_2\times C_2$. By \cite[Theorem 5.7]{nm gen bieb}, we have $d(\Gamma)\leq n$. Lastly, we assume $xy$ is of order $k$, where $k\geq 3$. We can rewrite the generating set of $\Gamma$ to be $\{\iota(e_1),...,\iota(e_n),\alpha\beta,\beta\}$. Define $\Gamma''=\langle \iota(e_1),...,\iota(e_n),\alpha\beta\rangle $, which is an $n$-dimensional Bieberbach subgroup of $\Gamma$ with holonomy group isomorphic to $C_k$. By Theorem A, $d(\Gamma'')\leq n-1$. Therefore $d(\Gamma)\leq n$.
	\end{proof}

\end{document}